\documentclass[10pt,a4paper,reqno]{amsart}
\usepackage{amsfonts}
\usepackage{amssymb}
\usepackage{amsmath}
\usepackage{amsthm}
\usepackage{cite}
\usepackage{enumitem}
\usepackage{tikz}
\usepackage{subfigure}
\usepackage{stackrel}
\usepackage{hyperref} 
\usepackage{dsfont}
\usepackage{color}

\theoremstyle{plain}
\newtheorem{thm}{Theorem} 

\newtheorem{lem}[thm]{Lemma}
\newtheorem{prop}[thm]{Proposition}

\newtheorem{rem}[thm]{Remark}

\providecommand{\ind}{\mathds{1}} 
\providecommand{\sm}{\setminus}
\providecommand{\N}{\mathbb{N}}
\providecommand{\R}{\mathbb{R}} 
\providecommand{\Z}{\mathbb{Z}}

\providecommand{\eps}{\varepsilon}

\providecommand{\wto}{\rightharpoonup}
\providecommand{\skp}[2]{\langle#1,#2\rangle}
\providecommand{\les}{\lesssim}

\DeclareMathOperator{\supp}{supp}

\DeclareMathOperator{\loc}{loc}
\DeclareMathOperator{\sol}{sol}
\DeclareMathOperator{\curl}{curl}

\DeclareMathOperator{\spa}{span}

\renewcommand{\qed}{\hfill $\Box$}

\newcommand{\vecIII}[3]{
\ensuremath{
\begin{pmatrix}
#1 \\ #2 \\ #3 \\
\end{pmatrix}}}

\renewcommand{\r}[1]{\textcolor{red}{#1}}

\begin{document}

\allowdisplaybreaks

\title{Dual variational methods for static Nonlinear Maxwell's Equations}

\author{Rainer Mandel\textsuperscript{1}}


\date{\today}   

\begin{abstract}
  We prove the existence of a ground state and infinitely many geometrically distinct solutions for
  static nonlinear Maxwell's equations on $\R^3$. Our existence result relies on a variant of the Symmetric
  Mountain Pass Theorem that applies to periodic as well as vanishing nonlinearities. 
  It is applied in a dual variational setting and thus provides an  alternative approach with respect to the
  direct variational method introduced by Mederski.
\end{abstract}

\maketitle
\allowdisplaybreaks
\setlength{\parindent}{0cm}

\section{Introduction}
    
In this note we provide a dual variational framework for nonlinear Maxwell's equations of the form 
\begin{equation}\label{eq:NLCurlCurl}
  \nabla\times \mu(x)^{-1} \nabla\times E =  f(x,E) \qquad\text{in }\R^3.
\end{equation}
Here, $E:\R^3\to\R^3$ stands for the electric field, $\mu(x)\in\R^{3\times 3}$ is the permeability matrix
and the nonlinearity $f(x,E)\in\R^3$ represents the electric displacement field
within the propagation medium, see~\cite[pp. 825-826]{Med_GS}. In the past years, several existence results
for nontrivial solutions of \eqref{eq:NLCurlCurl} have been found. Most of them 
rely on variational methods so that $f(x,E) = \partial_E F(x,E)$ is assumed for some
scalar-valued function $F$. 
These results, for $\mu(x)=I_{3\times 3}$, may be separated into two classes: the first deals with 
cylindrically symmetric solutions that are automatically divergence-free. Cylindrical symmetry means that, up to permutation of
coordinates, the solution has the form
$$
  E(x) = u\big(|(x_1,x_2)|,x_3\big) \vecIII{-x_2}{x_1}{0} 
  \qquad (x\in\R^3).
$$
The curl-curl operator acts like the classical Laplacian on such functions so that several standard tools 
from elliptic PDEs such as the local compactness of Sobolev embeddings can be used in a straightforward
manner.
Of course, the nonlinearity $f$ then needs to be cylindrically symmetric with respect to $x$ as well.
The first contribution in this direction is due to Azzollini, Benci, D'Aprile, Fortunato
\cite[Theorem~1]{AzzBenDApFor_Static} in the case of constant coefficients. A model nonlinearity for their
and all subsequent existence results is given by
\begin{equation*}
  f(x,E) = \min\{|E|^{p-2},|E|^{q-2}\}E \qquad\text{where } 2<p<6<q<\infty.
\end{equation*}
Equations of the form~\eqref{eq:NLCurlCurl} with more general
nonlinearities and cylindrically symmetric $x$-dependence were discussed  in
\cite{DApSic_Magnetostatic,BaDoPlRe_GroundStates,HirRei_Cylindrical}.
The second class of papers~\cite{Med_GS,MedSchSzu_Multiple} by Mederski and coauthors deals with
$\Z^3$-periodic nonlinearities where a cylindrically symmetric ansatz does not make sense.
The variational approach developed in these papers is very advanced and many difficulties have to be
overcome to prove the existence of a nontrivial solutions via some detailed analysis
of generalized Palais-Smale sequences for the energy functional.
Our goal is to show that Maxwell's  equations admit a convenient dual formulation that allows to prove
existence results via more standard critical point theorems, notably the Symmetric Mountain Pass Theorem
\cite{Willem,AmbRab}. In order to deal with vanishing and periodic nonlinearities at the same time, we
use a variant of this theorem from~\cite{Man_TimeHarmonic}. 
  
 \medskip
  
We shall use the following assumptions on the data: 
  \begin{itemize}
    \item[(A1)] $\mu\in L^\infty(\R^3;\R^{3\times 3})$ is symmetric and uniformly positive definite.   
  \item[(A2)] $f:\R^3\times\R^3\to\R^3$ is a Carath\'{e}odory function with $f(x,E)=f_0(x,|E|)|E|^{-1}E$
  where $s\mapsto s^{-1}f_0(x,s)$ is positive, increasing and piecewise continuously differentiable on $(0,\infty)$.
  There are exponents $p\in (2,6),q\in [p,\infty)$ and positive functions $\Gamma_1,\Gamma_2\in
  L^\infty(\R^3)$ such that for almost all $x\in\R^3$ and all $s>0$
    \begin{align*}
    \frac{1}{2}f_0(x,s)s - \int_0^s f_0(x,t)\,dt  \geq  \min\Big\{\Gamma_1(x) s^p,\Gamma_2(x) s^q\Big\}
    \sim s^2 \partial_s f_0(x,s).  
  \end{align*}
\end{itemize}
We say that $(A)_{per}$ holds if in addition to (A1),(A2) we have $q>6$ and the functions $x\mapsto
\mu(x)$ and $x\mapsto f(x,E)$ are $\Z^3$-periodic for all $E\in\R^3$. In that situation two solutions of
\eqref{eq:NLCurlCurl} are said to be geometrically distinct if they are not $\Z^3$-translates of each other. 
We say that $(A)_{van}$ holds if in addition to (A1),(A2) one of the following conditions hold:
\begin{equation}\label{eq:van} 
   \begin{cases}
  \;q>6\quad\text{and}\quad \Gamma_1(x)\to 0 \text{ or }\Gamma_2(x)\to 0 \text{ as }|x|\to\infty, \\ 
  \;q= 6\quad\text{and}\quad \Gamma_2(x)\to 0\text{ as } |x|\to\infty, \\ 
  \;q<6\quad\text{and}\quad \Gamma_2(x)\to 0\text{ as } |x|\to\infty \text{  and  } \Gamma_2\in L^s(\R^3)
  \text{ for some } s<\frac{6}{6-q}.
  \end{cases}
\end{equation}  
In that case, geometrically distinct solutions are nothing but distinct solutions. 
A bound state is a nontrivial critical point of the associated energy functional 
$$
    I(E):= \frac{1}{2} \int_{\R^3} \mu(x)^{-1}(\nabla\times E)\cdot (\nabla\times E)\,dx - \int_{\R^3}
    F(x,E)\,dx \qquad\text{where }F(x,E):=\int_0^{|E|} f_0(x,s)\,ds.
$$ 
over all measurable functions $E$ satisfying $\nabla\times E\in L^2(\R^3;\R^3)$ and $F(\cdot,E)\in L^1(\R^3)$. 
A bound state with least energy among all bound states is called a ground state. 
Our main result reads as follows:
     
  \begin{thm} \label{thm:main} 
    Assume $(A)_{van}$ or $(A)_{per}$.  Then~\eqref{eq:NLCurlCurl} has a ground state and infinitely many geometrically
    distinct bound states.
  \end{thm}

This appears to be the first existence result involving non-constant $\mu$ or vanishing coefficients given
that the assumptions on the nonlinearity in the papers~\cite{Med_GS,MedSchSzu_Multiple} are not satisfied.
 Nevertheless, the most interesting outcome of this paper might be that the dual formulation has some
 technical advantages compared to the standard variational approach and that the critical point theorem 
 from~\cite{Man_TimeHarmonic} applies to vanishing and periodic nonlinearities.  

\medskip
 
  \begin{rem} ~
    \begin{itemize}
      \item[(a)] Dual variational methods for non-static time-harmonic
      Maxwell's equations can be found in~\cite{Man_TimeHarmonic}. The function spaces and the
      restrictions on the nonlinearities are different in the static case.  
      \item[(b)] If one of the nonnegative functions $\Gamma_1,\Gamma_2$ vanishes on some nonempty open
      set, then the ground state level is zero and the set of ground states is given by all gradients
      with support in $\{\Gamma_1=0\}\cup\{\Gamma_2=0\}$. This follows from the fact that a
      critical point at energy level $0$ satisfies 
      $$
        0 
        =  I(E)- \frac{1}{2} I'(E)[E]
        = \int_{\R^3} F(x,E)-\frac{1}{2}f(x,E)\cdot E \,dx
        \gtrsim \int_{\R^3}  \min\Big\{\Gamma_1(x) |E|^p,\Gamma_2(x) |E|^q\Big\} \,dx
      $$ 
      by assumption (A2). It is not known whether ground states exist for sign-changing $\Gamma_1,\Gamma_2$
      where the dual variational method leads to the study of a strongly indefinite functional, 
      cf.~\cite{ManYes_DualSignchanging}. So, in contrast to the case $\Gamma_1,\Gamma_2>0$, passing to the
      dual formulation does not seem to simplify the analysis. The existence
      of ground states and infinitely many bound states remains open in this case.
      \item[(c)] Theorem~\ref{thm:main}  covers subcritical power-type nonlinearities with decaying
      coefficients like 
      $$
        f(x,E) = \Gamma(x)|E|^{p-2}E \quad\text{where}\quad 2<p<6,\; 
        0<\Gamma(x)\les (1+|x|)^{-\frac{6-p}{2}-\sigma},\;\sigma>0.
      $$
    \end{itemize}
  \end{rem}

\medskip

  In the following   the symbol $\les$ stands for $\leq C$ for some positive number $C$, similarly for
 $\gtrsim$. We write $A\sim B$ if $A\les B$ and $B\les A$. The exponents  $p',q'$ denote, as usual, the
 H\"older conjugates given by $\frac{1}{p}+\frac{1}{p'}=\frac{1}{q}+\frac{1}{q'}=1$. We fix the standard norm
 $\|\cdot\|_r$ on $L^r(\R^3;\R^3)$. 
 The divergence operator $\nabla\cdot$ and the curl operator
 $\nabla\times $ are understood in the distributional sense.  We shall use that the
 curl operator is symmetric on $C_0^\infty(\R^3;\R^3)$, which follows from the
 identity 
 $$ 
   \nabla\cdot (f\times g) = (\nabla\times f)\cdot g - f \cdot (\nabla\times g).
 $$

 \section{An equivalent dual formulation}
  
  We  use a dual variational approach to prove Theorem~\ref{thm:main}. This means that instead of considering
  the electric field $E$ as the unknown, we treat~\eqref{eq:NLCurlCurl} as a variational problem for the
  polarization vector field $P(x):= f(x,E(x))$. This leads to the equation
   \begin{equation}\label{eq:NLCurlCurldual}
    \mathcal L_\mu(\psi(x,P)) = P \qquad\text{in
    }\R^3\qquad\text{where}\quad
    \mathcal L_\mu(E):= \nabla\times \mu(x)^{-1} \nabla\times E. 
  \end{equation}  
  The function $\psi(x,\cdot)$ denotes the inverse of $f(x,\cdot)$.   
  We shall prove that $\psi$ exists with the following properties:
 \begin{itemize}
  \item[(A2')]  $\psi:\R^3\times\R^3\to\R^3$ is a Carath\'{e}odory function with  
  $\psi(x,P)=\psi_0(x,|P|)|P|^{-1}P$ where 
  $z\mapsto z^{-1}\psi_0(x,z) \text{ is positive, decreasing and piecewise continuously differentiable on
  }(0,\infty)$.
  There are exponents $p\in (2,6), q\in [p,\infty)$ and positive functions
  $\Gamma_1,\Gamma_2\in L^\infty(\R^3)$ such that for almost all $x\in\R^3$ and all $z>0$
  \begin{align*}
    \int_0^z \psi_0(x,s)\,ds -\frac{1}{2} \psi_0(x,z)z 
    \gtrsim z \max\{ (\Gamma_1(x)^{-1}z)^{p'-1},(\Gamma_2(x)^{-1}z)^{q'-1}\}
    \sim z^2 \partial_z \psi_0(x,z).       
  \end{align*}
\end{itemize} 
Again, we say that $(A')_{per}$ holds if in addition to (A1),(A2') we have $q>6$ and the functions $x\mapsto
\mu(x)$ and $x\mapsto f(x,z)$ is $\Z^3$-periodic on $\R^3$ for all $z>0$. Similarly, $(A')_{van}$ holds if in
addition to (A1),(A2') one of the conditions in~\eqref{eq:van} holds.
 
 \begin{prop}\label{prop:Apsi}
    Assume that $f:\R^3\times\R^3\to\R^3$ satisfies~(A2). Then
    $\psi(x,\cdot):=f(x,\cdot)^{-1}$ exists for almost all $x\in\R^3$ and satisfies (A2'). In particular,
    $(A)_{van},(A)_{per}$ implies $(A')_{van},(A')_{per}$, respectively.
  \end{prop}
  \begin{proof}
    By assumption (A2), for almost all $x\in\R^3$ the function $z\mapsto f_0(x,z)$ is
    positive, continuous and piecewise continuously differentiable  on $(0,\infty)$ with positive derivative
    as well as $f_0(x,z)\to 0$ as $z\to 0$ and $f_0(x,z)\to +\infty$ as $z\to\infty$. So
    $f_0(x,\cdot):[0,\infty)\to [0,\infty)$ admits a positive and piecewise continuously differentiable  
    inverse $\psi_0(x,\cdot):=f_0(x,\cdot)^{-1}$ with positive derivative for such $x\in\R^3$.
    Moreover, $z\mapsto z^{-1} \psi_0(x,z)$ is decreasing on $(0,\infty)$ given that $s\mapsto s^{-1}
    f_0(x,s)$ is increasing on $(0,\infty)$. This implies $f(x,\cdot)^{-1}=\psi(x,\cdot)$ where
    $\psi(x,P):=\psi_0(x,|P|)|P|^{-1}P$. Assumption (A2) implies, for $z:=f_0(x,s)$ and $s=\psi_0(x,z)$, 
	\begin{align*}
	  z \sim \min\Big\{\Gamma_1(x) s^{p-1},\Gamma_2(x) s^{q-1}\Big\} \quad\text{and hence}\quad
	  s  \sim  \max\Big\{(\Gamma_1(x)^{-1}z)^{p'-1},(\Gamma_2(x)^{-1}z)^{q'-1}\Big\}.
	\end{align*}
	This gives
	\begin{align*}
	  \partial_z \psi_0(x,z) 
	  &= \frac{1}{\partial_s f_0(x,s)} 
	  \sim \frac{1}{\min\{\Gamma_1(x) s^{p-2},\Gamma_2(x)s^{q-2}\}} 
	  \sim \frac{s}{z}
	  \sim \max\Big\{ \Gamma_1(x)^{1-p'}z^{p'-2},\Gamma_2(x)^{1-q'}z^{q'-2}\Big\}.  
	\end{align*}
	Furthermore, by differentiation we find the identity
	$$
	  \int_0^{f_0(x,s)} \psi_0(x,t)\,dt + \int_0^s f_0(x,t)\,dt = sf_0(x,s)
	  \qquad\text{for all }s\geq 0
	$$
	and conclude 
	\begin{align*}
	  \int_0^z \psi_0(x,t)\,dt-\frac{1}{2} \psi_0(x,z)z
	  &=\frac{1}{2} f_0(x,s)s - \int_0^s f_0(x,t)\,dt \\
	  &\gtrsim \min\Big\{\Gamma_1(x) s^{p},\Gamma_2(x) s^{q}\Big\}     \\
	  &\sim  sz \\
	  &\sim \max\Big\{  \Gamma_1(x)^{1-p'} z^{p'},\Gamma_2(x)^{1-q'} z^{q'}\Big\}.
	\end{align*}
	So (A2') holds and the claim is proved.
  \end{proof}
 
  So the task is to find suitable solutions $P$ of \eqref{eq:NLCurlCurldual} for functions $\psi$ satisfying
  $(A')_{van}$ or $(A')_{per}$. The dual method requires to invert the linear differential operator in a
  suitable sense. To do this, we need appropriate function spaces. Given that $\mathcal L_\mu(E)=\mathcal
  L_\mu(E+\nabla\phi)$ for all $\phi\in C_0^\infty(\R^3)$, it is mandatory to look for uniquely determined
  solutions of \eqref{eq:NLCurlCurldual} within divergence-free functions.  Define   
 \begin{align*}
   H&:= \Big\{E\in L^1_{\loc}(\R^3;\R^3): \nabla\times E \in L^2(\R^3;\R^3),
   \;\nabla\cdot E=0 \Big\}, \\
   \skp{E}{\tilde E}_\mu
   &:=   \int_{\R^3} \mu(x)^{-1}(\nabla\times E)\cdot (\nabla\times \tilde E)\,dx.
 \end{align*}
 The corresponding norm is $\|\cdot\|_\mu$ and we omit the index if $\mu(x)=I_{3\times 3}$ is the identity
 matrix. In this case we have $\mathcal L_\mu = -\Delta$ on $H$ and, accordingly, $\mathcal L_\mu^{-1} P = (-\Delta)^{-1}P 
    =  K\ast P$  where $K(z)= \frac{1}{4\pi|z|}$.

 \begin{prop} \label{prop:embeddingH}
   Assume (A1). Then $(H,\skp{\cdot}{\cdot})$ is a Hilbert space continuously embedded
   into $L^6_{\sol}(\R^3;\R^3)$ and compactly embedded into $L^r_{\sol}(B;\R^3)$ for bounded balls
   $B\subset\R^3$ with $r<6$ where, for measurable $\Omega\subset\R^3$, 
   $$
     L^p_{\sol}(\Omega;\R^3):= \big\{ h\in L^p(\Omega;\R^3): \nabla\cdot h=0\big\}. 
   $$
 \end{prop}
 \begin{proof}
   We omit the proof of the Hilbert space property. The embedding into $L^6_{\sol}(\R^3;\R^3)$ is a consequence
   of Sobolev's Embedding Theorem $\dot H^1(\R^3)\hookrightarrow L^6(\R^3)$ because of  
   $$
     \|E\|_\mu^2 
     = \int_{\R^3} \mu(x)^{-1}(\nabla\times E)\cdot (\nabla\times E)\,dx
     \sim \int_{\R^3} |\nabla\times E|^2\,dx
     \sim \int_{\R^3} |\nabla E|^2\,dx.
   $$    
 \end{proof}
 
 To solve~\eqref{eq:NLCurlCurldual} in $H$ we use the Lax-Milgram Lemma or the Riesz Representation Theorem. 
 
  
 \begin{prop}\label{prop:Lmu}
   Assume (A1). Then, for all $P\in L^{6/5}_{\sol}(\R^3;\R^3)$, the linear boundary value problem 
   $$ 
     \int_{\R^3} \mu(x)^{-1}(\nabla\times E)\cdot (\nabla\times \tilde E)\,dx 
     = \int_{\R^3} P\cdot\tilde E\,dx
     \qquad\text{for all }\tilde E\in H
   $$
   has a unique solution $E\in H$, denoted by $\mathcal L_\mu^{-1}P$. The selfdual operator
   $$
     L^{6/5}_{\sol}(\R^3;\R^3)\to L^6_{\sol}(\R^3;\R^3),\qquad P\mapsto \mathcal L_\mu^{-1}P
   $$ 
   is symmetric with 
   \begin{equation} \label{eq:Lmu}
     \|\mathcal L_\mu^{-1}P\|_\mu^2 = \int_{\R^3} \mathcal L_\mu^{-1}P\cdot P\,dx,\qquad  
     \|\mathcal L_\mu^{-1}P\|_\mu \les \|P\|_{\frac{6}{5}},\qquad  
     \int_{\R^3} \mathcal L_\mu^{-1}P\cdot P \,dx 
     \sim  \int_{\R^3} (-\Delta)^{-1}P\cdot P \,dx. 
   \end{equation}
 \end{prop}
 \begin{proof}
   We concentrate on the last estimate in \eqref{eq:Lmu}. Plugging in the test functions 
   $\tilde E := (-\Delta)^{-1}P$ into the weak PDE solved by $E:=\mathcal L_\mu^{-1}P$, we find
   using (A1)
 \begin{align*}
   \int_{\R^3} (-\Delta)^{-1} P\cdot P\,dx
   &= \int_{\R^3} P\cdot \tilde E \,dx 
   \;=\; \int_{\R^3} \mu(x)^{-1}(\nabla\times E)\cdot (\nabla\times \tilde E)\,dx \\
   &= \skp{E}{\tilde E}_\mu 
   \;\leq\; \|E\|_\mu\|\tilde E\|_\mu 
   \;\les\; \|E\|_\mu\|\tilde E\|  \\
   &\stackrel{\eqref{eq:Lmu}}= 
   \sqrt{\int_{\R^3} \mathcal L_\mu^{-1}P\cdot P\,dx}\cdot \sqrt{\int_{\R^3} (-\Delta)^{-1} P\cdot P\,dx}. 
 \end{align*} 
 The reverse estimate is proved analogously. 
  \end{proof}
 
 In view of Proposition~\ref{prop:Lmu} we have to find $P\in L^{6/5}_{\sol}(\R^3;\R^3)$ such that
 $\psi(x,P)-\mathcal L_\mu^{-1} P$ is a gradient. This turns out to be the Euler-Lagrange equation of the
 functional 
 \begin{equation} \label{eq:functionalJdual}
    J(P) := \int_{\R^3} \Psi(x,P)\,dx - \frac{1}{2}\int_{\R^3} \mathcal L_\mu^{-1}P\cdot P\,dx,\qquad  
    \Psi(x,E):=\int_0^{|E|} \psi_0(x,s)\,ds
  \end{equation}
  over a suitable function space $X\subset L^{6/5}_{\sol}(\R^3;\R^3)$. 
  We have to make sure that each integral is well-defined for $P\in X$. 
  To find a reasonable candidate for $X$  we note 
 \begin{align}  \label{eq:Zproperty}
   \begin{aligned}
   \int_{\R^3}\Psi(x,P)\,dx<\infty
   &\quad\stackrel{(A2')}\Leftrightarrow\quad
   \int_{\R^3}\max\Big\{  \Gamma_1(x)^{1-p'} |P|^{p'},\Gamma_2(x)^{1-q'} |P|^{q'}\Big\}\,dx<\infty \\
   &\quad\Leftrightarrow\quad
   P\in Z:= \Gamma_1^{1/p}L^{p'}(\R^3;\R^3)\cap \Gamma_2^{1/q}L^{q'}(\R^3;\R^3).
 \end{aligned}
 \end{align} 
 Then $Z$ is a reflexive Banach space equipped with the norm
 \begin{equation}\label{eq:def_norm}
   \|P\| := \|\Gamma_1^{-\frac{1}{p}} P\|_{p'}+\|\Gamma_2^{-\frac{1}{q}} P\|_{q'}.
 \end{equation}  
 We then define the subspace
 $$
   X:= Z \cap \big(\nabla\times L^2(\R^3;\R^3)\big) 
   = \Big\{P\in Z: P=\nabla\times G \text{ for some } G\in L^2(\R^3;\R^3)\Big\}.
 $$

%

 \begin{lem} \label{lem:ReflexiveBanachSpace}
   Assume $(A')_{per}$ or $(A')_{van}$.
   Then $X$ is a reflexive Banach space that is continuously embedded into $L^{6/5}_{\sol}(\R^3;\R^3)$.
 \end{lem} 
 \begin{proof}
   We first prove the embedding and start with the case $q\geq 6$. We then choose $\theta\in (0,1]$ such that
   $\frac{5}{6}= \frac{1-\theta}{p'}+\frac{\theta}{q'}$. Then we have for $P\in X$
    $$
      \|P\|_{\frac{6}{5}}
      \leq \|P\|_{p'}^{1-\theta}\|P\|_{q'}^\theta 
      \les \|P\|_{p'} + \|P\|_{q'} 
      \les \|\Gamma_1^{-\frac{1}{p}}P\|_{p'}+  \|\Gamma_2^{-\frac{1}{q}}P\|_{q'} 
      = \|P\|. 
    $$
    In the case $q<6$, which only occurs under assumption $(A')_{van}$, we have~\eqref{eq:van} and thus
    $\Gamma_2\in L^{\frac{6}{6-q}}(\R^3)$. Hence,   
    \begin{align*}
      \|P\|_{\frac{6}{5}}
      \leq \|\Gamma_2^{-\frac{1}{q}}P\|_{q'} \|\Gamma_2^{\frac{1}{q}}\|_{\frac{6q}{6-q}}
      \les \|P\|.  
    \end{align*} 
    
   \medskip
 
   We now prove the reflexivity. It suffices to show that $X$ is a closed subspace of the reflexive Banach
   space $Z$.
   So assume $P_n\to P$ in $Z$ with $P_n\in X$, set $E_n:=\mathcal L_\mu^{-1}(P_n),E:=\mathcal L_\mu^{-1}(P)$. The
   above estimate implies $P_n\to P$ in $L^{6/5}_{\sol}(\R^3;\R^3)$. 
   So Proposition~\ref{prop:Lmu} implies $E_n\to E$ in $H$   and thus, by Proposition~\ref{prop:embeddingH},
   in $L_{\sol}^6(\R^3;\R^3)$. This ensures 
   $$
     \infty
     >\int_{\R^3} \mathcal L_\mu^{-1}P\cdot P\,dx
     = \int_{\R^3} E\cdot P\,dx
     = \lim_{n\to\infty} \int_{\R^3} E_n\cdot P_n\,dx
     = \lim_{n\to\infty} \int_{\R^3} \mathcal L_\mu^{-1}(P_n) \cdot P_n\,dx. 
   $$ 
   From $P_n\in X$ we get $P_n=\nabla\times G_n$ for some $G_n\in L^2(\R^3;\R^3)$.   
   Using the classical Helmholtz Decomposition of $L^2(\R^3;\R^3)$, we may w.l.o.g. assume that $G_n$ is
   divergence-free because the curl-operator vanishes on gradients. So    
  $\nabla\times (\mu^{-1} \nabla\times E_n-G_n)= \mathcal L_\mu(E_n) - P_n = 0$ holds in the weak sense and
  thus $\mu^{-1}\nabla\times E_n = G_n + \nabla \psi_n$
  for some $\psi_n\in \dot H^1(\R^3)$. Hence,
  \begin{align*}
     \infty
     &>
     \lim_{n\to\infty} \int_{\R^3} \mathcal L_\mu^{-1}(\nabla\times G_n) \cdot (\nabla\times G_n)\,dx \\
     &\stackrel{\eqref{eq:Lmu}}= \lim_{n\to\infty} \| \mathcal L_\mu^{-1}(\nabla\times G_n)\|_\mu^2  
     \;=\; \lim_{n\to\infty} \|E_n\|_\mu^2  \\
     &=  \lim_{n\to\infty} \int_{\R^3} \mu^{-1}(\nabla\times E_n)\cdot (\nabla\times E_n)  \,dx    
     \;\sim\;  \lim_{n\to\infty} \int_{\R^3} |\mu^{-1}\nabla\times E_n|^2  \,dx    \\
     &= \lim_{n\to\infty} \int_{\R^3} |G_n+\nabla\psi_n|^2\,dx        
     \;=\; \lim_{n\to\infty} \int_{\R^3} |G_n|^2+|\nabla\psi_n|^2 \,dx.  
  \end{align*}  
   So $(G_n)$ is bounded in $L^2(\R^3;\R^3)$ and we may select a subsequence, again denoted by $(G_n)$,
   that converges weakly   to some $G\in L^2(\R^3;\R^3)$.
   This function satisfies, for all $\phi\in C_0^\infty(\R^3;\R^3)$,
   \begin{align*}
     \int_{\R^3} P \cdot\phi\,dx 
     &= \lim_{n\to\infty} \int_{\R^3} P_n \cdot\phi\,dx
     = \lim_{n\to\infty} \int_{\R^3} (\nabla\times G_n) \cdot\phi \,dx \\
     &= \lim_{n\to\infty} \int_{\R^3} G_n \cdot (\nabla\times \phi)\,dx
     =  \int_{\R^3} G  \cdot (\nabla\times \phi)\,dx.
  \end{align*}
  This implies $P=\nabla\times G$ in the distributional sense, so $P\in X$. So $X$ is a closed
  subspace of $Z$.
 \end{proof}

  \begin{prop}\label{prop:J}
    Assume $(A')_{per}$ or $(A')_{van}$. Then  we have $J\in C^1(X)$ with
    $$
      J'(P)[\tilde P] = \int_{\R^3} \psi(x,P)\cdot \tilde P\,dx 
      - \int_{\R^3} \mathcal L_\mu^{-1}P\cdot \tilde  P\,dx 
    $$
    for all $P,\tilde P\in X$. Moreover, $J'(P) = 0$ holds if and only if, in the sense of distributions,
    $$
      \nabla\times \big(\psi(x,P)-\mathcal L_\mu^{-1}P\big)=0.
    $$   
  \end{prop}
  \begin{proof}     
    The combination of Proposition~\ref{prop:Lmu} and Proposition~\ref{prop:embeddingH} implies that the
    quadratic term in $J$ is well-defined and hence smooth on $X$. Moreover, \eqref{eq:Zproperty} shows that the first
    integral is finite for $P\in X$. We skip the proof of continuous differentiability. A critical point
    $P\in X$ of $J$ is characterized by
    $$
      \int_{\R^3} \big(\psi(x,P)- \mathcal L_\mu^{-1}P\big)\cdot \tilde  P\,dx = 0
      \qquad\text{for all }\tilde P\in X.
    $$
    In particular, this identity holds for $\tilde P=\nabla\times \phi$ where $\phi\in
    C_0^\infty(\R^3;\R^3)$ is arbitrary. This implies the characterization of critical points  and we
    deduce the result.
  \end{proof}

\begin{lem} \label{lem:equivalence}
  Assume $(A)_{van}$ or $(A)_{per}$.
 Then the following two statements for $P,E$ are equivalent where $P= f(x,E)$, $E=\psi(x,P)$ with $\psi$ as
 in Proposition~\ref{prop:Apsi}:
  \begin{itemize}
    \item[(i)] $I'(E)=0$ where $\nabla\times E\in L^2(\R^3;\R^3)$ and $F(\cdot,E)\in L^1(\R^3)$.
    \item[(ii)] $J'(P)=0$ where $P\in \nabla\times \big(L^2(\R^3;\R^3)\big)$ and $\Psi(\cdot,P)\in
    L^1(\R^3)$, i.e., $P\in X$.
  \end{itemize}    
\end{lem}
\begin{proof}
  Assume first (i). Then   Proposition~\ref{prop:Apsi} imply, thanks to
  $|E|=\psi_0(x,|P|),|P|=f_0(x,|E|)$,
  \begin{align*}
    \infty
    &>\int_{\R^3} F(x,E)\,dx
    \;=\; \int_{\R^3} \int_0^{|E|} f_0(x,s)\,ds \,dx 
    \;\sim\; \int_{\R^3} f_0(x,|E|)|E| \,dx \\  
    &\sim\; \int_{\R^3} |P|\psi_0(x,|P|)  \,dx 
    \;\sim\; \int_{\R^3} \int_0^{|P|} \psi_0(x,z)\,dz \,dx 
    \;\sim\; \int_{\R^3} \Psi(x,P)\,dx.
  \end{align*}
  Moreover, $I'(E)=0$ implies $\nabla\times\mu(x)^{-1}\nabla\times E = P$ in the weak sense and thus 
  $P=\nabla\times G$ with $G:= \mu(x)^{-1}\nabla\times E\in L^2(\R^3;\R^3)$. Moreover, given
  Proposition~\ref{prop:Lmu}, there is $\psi\in \dot H^1(\R^3)$ such that $E=\mathcal
  L_\mu^{-1}(P)+\nabla\psi$. 
  So $E=\psi(x,P)$ leads to
  $$
    \nabla\times \big(\psi(x,P)-\mathcal L_\mu^{-1}P\big)
    = \nabla\times (\nabla\psi) =0.
  $$ 
  By Proposition~\ref{prop:J}, this means $J'(P)=0$ and we deduce (ii). 
  
  \medskip
  
  Now assume (ii). As above, we find that $F(\cdot,E)$ is integrable where $E:=\psi(x,P)$. Moreover, 
  $J'(P)=0$ implies $\nabla\times \big(E-\mathcal L_\mu^{-1} P\big)=0$. 
  Since $P\in L^{6/5}_{\sol}(\R^3;\R^3)$, we know $\mathcal L_\mu^{-1}P\in H$ and thus 
  $$
    \nabla\times E = \nabla\times \big(\mathcal L_\mu^{-1}P\big) \in L^2(\R^3;\R^3).
  $$ 
  Since   $\mathcal L_\mu^{-1}P$ is orthogonal to gradients, we obtain
  for all $\phi\in C_0^\infty(\R^3;\R^3)$ 
  \begin{align*}
    \int_{\R^3} \mu(x)^{-1} (\nabla\times E) \cdot (\nabla\times\phi)\,dx
    &= \int_{\R^3} \mu(x)^{-1} (\nabla\times \mathcal L_\mu^{-1}P) \cdot (\nabla\times\phi)\,dx \\
    &= \int_{\R^3}  P  \cdot \phi \,dx \\
    &= \int_{\R^3} f(x,E) \cdot \phi\,dx.
  \end{align*}
  This proves $I'(E)=0$ and thus (i).
\end{proof}
 
  As a consequence, critical points of $I$ are in one-to-one correspondence with critical points of $J$. So,
  essentially, it will be sufficient to prove the existence of infinitely many geometrically critical points
  of the latter. This will be done in the next section. In our final section we briefly collect the relevant
  material in order to deduce Theorem~\ref{thm:main}.

\section{Analysis of the dual problem}
  
   In view of Lemma~\ref{lem:equivalence} we want to prove the existence a ground state and infinitely many
   geometrically distinct bound states of the energy functional
 \begin{equation*}
    J:X\to\R,\qquad  J(P) = J_1(P) - \frac{1}{2}\int_{\R^3} \mathcal L_\mu^{-1}P\cdot P\,dx,\qquad 
    J_1(P) := \int_{\R^3} \Psi(x,P)\,dx
  \end{equation*}
  where $X$ was constructed in Lemma~\ref{lem:ReflexiveBanachSpace}. Throughout
  this section we assume that  $(A')_{per}$ or $(A')_{van}$ holds.

\begin{prop} \label{prop:estimates}
  Assume $(A')_{per}$ or $(A')_{van}$. Then
   $$
      \min\{ \|P\|^{p'},\|P\|^{q'}\}  
      \les J_1(P) 
      \les \max\{\|P\|^{p'},\|P\|^{q'}\}.  
    $$
\end{prop}
\begin{proof}
  The starting point is the following
 \begin{align*}
    J_1(P)
    &= \int_{\R^3} \left(\int_0^{|P|}  \psi_0(x,s)  \,ds\right)\,dx \\ 
    &\stackrel{(A2')}\sim  \int_{\R^3} \int_0^{|P|}    \max\big\{ (\Gamma_1(x)^{-1}
    s)^{p'-1},(\Gamma_2(x)^{-1}s)^{q'-1}\big\} \,ds  \,dx \\
     &\sim  \int_{\R^3}  \Gamma_1(x)^{1-p'}|P|^{p'} + \Gamma_2(x)^{1-q'}|P|^{q'} \,dx \\
     &= \|\Gamma_1^{-\frac{1}{p}}P\|_{p'}^{p'} + \|\Gamma_2^{-\frac{1}{q}}P\|_{q'}^{q'}.
  \end{align*}  
  As a consequence, we have by definition of the norm 
  \begin{align*}
    J_1(P)    
    \stackrel{\eqref{eq:def_norm}}\les \|P\|^{p'}+\|P\|^{q'}
    \les \max\{\|P\|^{p'},\|P\|^{q'}\}. 
  \end{align*}
  On the other hand, for    $A(t):= \max\{t^{p'/q'},t^{q'/p'}\}$,
  \begin{align*}
    \|P\|^{p'}+\|P\|^{q'}
    &\stackrel{\eqref{eq:def_norm}}\sim 
     \|\Gamma_1^{-\frac{1}{p}} P\|_{p'}^{p'}+\|\Gamma_2^{-\frac{1}{q}} P\|_{q'}^{p'} 
    + \|\Gamma_1^{-\frac{1}{p}} P\|_{p'}^{q'}+\|\Gamma_2^{-\frac{1}{q}} P\|_{q'}^{q'} \\
    &\les  J_1(P) + J_1(P)^{\frac{p'}{q'}} +  J_1(P)^{\frac{q'}{p'}} \\
    &\les  A\left(J_1(P)  \right).  
  \end{align*}
  This implies
  \begin{align*} 
    J_1(P) 
    &\gtrsim A^{-1}\left(\|P\|^{p'}+\|P\|^{q'}\right) \\ 
    &\gtrsim  \min\Big\{ (\|P\|^{p'}+\|P\|^{q'})^{\frac{p'}{q'}},(\|P\|^{p'}+\|P\|^{q'})^{\frac{q'}{p'}}\Big\}
    \\
    &\gtrsim  \min\{  \|P\|^{p'},\|P\|^{q'}\}
  \end{align*}
  and the claim is proved. 
\end{proof}

\begin{prop} \label{prop:psi_estimate}
  Assume $(A')_{per}$ or $(A')_{van}$. 
  Then we have for all $P,Q\in X$ with $(P,Q)\neq (0,0)$ and all measurable $\Omega\subset\R^3$
  $$
    \int_{\Omega} \big(\psi(x,P)-\psi(x,Q)\big)\cdot (P-Q)\,dx
    \gtrsim \|(P-Q)\ind_\Omega\|^2 \min\Big\{ (\|P\|+\|Q\|)^{p'-2},(\|P\|+\|Q\|)^{q'-2}\Big\}. 
  $$
\end{prop}
\begin{proof}
  We first prove the following inequality for almost all $x\in\Omega$ and all $P,Q\in\R^3$
  \begin{equation}\label{eq:convexity_ptws}
    \big(\psi(x,P)-\psi(x,Q)\big)\cdot (P-Q)  
    \gtrsim \max\Big\{\Gamma_1(x)^{1-p'} (|P|+|Q|)^{p'-2},\Gamma_2(x)^{1-q'} (|P|+|Q|)^{q'-2}\Big\}|P-Q|^2.
  \end{equation}
 To this end define $f(\tau):=  \psi(x,Q+\tau(P-Q))\cdot(P-Q)$. Then
  $$
    \big(\psi(x,P)-\psi(x,Q)\big)\cdot (P-Q)  = f(1)-f(0) = f'(\tau)
    \quad\text{for some }\tau\in (0,1).
  $$ 
  So it remains to estimate $f'(\tau)$ from below. To do this we recall  
  $\psi(x,\xi)=\psi_0^*(x,|\xi|)\xi$ where $\psi_0^*(x,z)=z^{-1}\psi_0(x,z)$ for almost all $x\in\R^3,z>0$. 
  From (A2') we get $\partial_z \psi_0^*(x,z)\leq 0$ and thus for $\xi_\tau:=Q+\tau(P-Q)$
  \begin{align*}
    f'(\tau) 
    &= D\psi(x,\xi_\tau)[P-Q]\cdot (P-Q) \\ 
    &= \psi_0^*(x,|\xi_\tau|)|P-Q|^2+ |\xi_\tau| \partial_z \psi_0^*(x,|\xi_\tau|)
    \big(|\xi_\tau|^{-1} \xi_\tau\cdot (P-Q)\big)^2     \\
    &\geq \big(\psi_0^*(x,|\xi_\tau|) + |\xi_\tau| \partial_z \psi_0^*(x,|\xi_\tau|)\big) 
    |P-Q|^2\\ 
    &= \partial_z \psi_0(x,z)|_{z=|\xi_\tau|}   |P-Q|^2\\
    &\stackrel{(A2')}\gtrsim  \max\Big\{\Gamma_1(x)^{1-p'}
    |\xi_\tau|^{p'-2},\Gamma_2(x)^{1-q'} |\xi_\tau|^{q'-2}\Big\}|P-Q|^2\\
    &\gtrsim \max\Big\{\Gamma_1(x)^{1-p'} (|P|+|Q|)^{p'-2},\Gamma_2(x)^{1-q'} (|P|+|Q|)^{q'-2}\Big\}|P-Q|^2.
  \end{align*}  
  Here we used $p',q'<2$. 
  
  \medskip
  
  From this we deduce  
  \begin{align*}
    \int_{\Omega} \big(\psi(x,P)-\psi(x,Q)\big)\cdot (P-Q)\,dx
    &\gtrsim \int_{\Omega}  \Gamma_1(x)^{1-p'} (|P|+|Q|)^{p'-2}|P-Q|^2\,dx \\ 
     &+ \int_{\Omega}   \Gamma_2(x)^{1-q'} (|P|+|Q|)^{q'-2}|P-Q|^2 \,dx.    
  \end{align*}
   To estimate the first term from below set $\tilde P:=\Gamma_1^{-\frac{1}{p}}P$ and $\tilde Q:=
   \Gamma_1^{-\frac{1}{p}}Q$. Then H\"older's inequality implies
    \begin{align*}
      \|\Gamma_1^{-\frac{1}{p}}(P-Q)\ind_\Omega\|_{p'}
      &= \|(\tilde P-\tilde Q)\ind_\Omega\|_{p'} \\
      &\leq \Big\||\tilde P-\tilde Q|(|\tilde P|+|\tilde Q|)^{\frac{p'-2}{2}} \ind_\Omega \Big\|_2
      \Big\| (|\tilde P|+|\tilde Q|)^{\frac{2-p'}{2}} \Big\|_{\frac{2p}{p-2}}   \\
      &\leq \left(\int_{\Omega} |\tilde P-\tilde Q|^2(|\tilde P|+|\tilde Q|)^{p'-2}\,dx\right)^{\frac{1}{2}}  
      \cdot \|   |\tilde P|+|\tilde Q| \|_{p'}^{\frac{2-p'}{2}}   \\
      &\leq \left(\int_{\Omega}  \Gamma_1(x)^{1-p'} (|P|+|Q|)^{p'-2}|P-Q|^2\,dx\right)^{\frac{1}{2}}  
      \cdot (\|\tilde P\|_{p'}+\|\tilde Q\|_{p'})^{\frac{2-p'}{2}}   \\
      &\stackrel{\eqref{eq:convexity_ptws}}\les  \left(\int_{\Omega}  \big(\psi(x,P)-\psi(x,Q)\big)\cdot
      (P-Q)\,dx\right)^{\frac{1}{2}} \cdot (\|P\| +\|Q\| )^{\frac{2-p'}{2}}.   
 \intertext{Analogously,}
      \|\Gamma_2^{-\frac{1}{q}}(P-Q)\ind_\Omega\|_{q'}      
      &\stackrel{\eqref{eq:convexity_ptws}}\les  \left(\int_{\Omega}  \big(\psi(x,P)-\psi(x,Q)\big)\cdot
      (P-Q)\,dx\right)^{\frac{1}{2}} \cdot (\|P\| +\|Q\| )^{\frac{2-q'}{2}}.   
    \end{align*}
  Squaring and summing up these two estimates we obtain  
  \begin{align*}
    \|(P-Q)\ind_\Omega\|^2
    &\sim \|\Gamma_1^{-\frac{1}{p}}(P-Q)\ind_\Omega\|_{p'}^2+
    \|\Gamma_2^{-\frac{1}{q}}(P-Q)\ind_\Omega\|_{q'}^2  \\
    &\les  \int_{\Omega} \big(\psi(x,P)-\psi(x,Q)\big)\cdot (P-Q) \,dx 
     \cdot \left(   (\|P\| +\|Q\| )^{2-p'} + (\|P\| +\|Q\| )^{2-q'} \right) \\
    &\sim  \int_{\Omega} \big(\psi(x,P)-\psi(x,Q)\big)\cdot (P-Q) \,dx 
     \cdot \max\Big\{ (\|P\| +\|Q\| )^{2-p'}, (\|P\| +\|Q\| )^{2-q'}\Big\}.
  \end{align*}
\end{proof}

  \begin{prop}\label{prop:J1}
    Assume $(A')_{per}$ or $(A')_{van}$. Then $J_1$ is weakly lower semicontinuous.
  \end{prop}
  \begin{proof}
   $J_1$ is convex, and even strictly convex, because of the estimate from
   Proposition~\ref{prop:psi_estimate}. Indeed,
  $$
    (J_1'(P)-J_1'(Q))[P-Q]
    = \int_{\R^3} \big(\psi(x,P)-\psi(x,Q)\big)\cdot (P-Q)\,dx 
    > 0
    \quad\text{for }P,Q\in X,\; P\neq Q.
  $$
  As a continuous convex functional $J_1$ is weakly lower semicontinuous.
\end{proof}

\subsection{Existence of infinitely many solutions}

Our strategy is, assuming $(A')_{per}$ or $(A')_{van}$, to verify the hypotheses of the following
Critical Point Theorem, which is a simplified version of \cite[Theorem~5]{Man_TimeHarmonic}.

\begin{thm}\label{thm:CPTheorem}
  Assume  
 \begin{itemize}
  \item[(G1)] $X$ is a Banach space.
  \item[(G2)] $J\in C^1(X)$ is even with $J(0)=0$ and, for some $\rho>0$, 
  $$
    \inf_{S_\rho} J>0
    \qquad\text{where}\qquad S_\rho = \{v\in X: \|v\|=\rho\}.
  $$
  \item[(G3)] For any given $m\in\N$ there is a finite-dimensional subspace $X_m\subset X$ such that
  $J(u)\to -\infty$ uniformly for $u\in
  X_m$ as $\|u\|\to\infty$ and $\dim(X_m)\nearrow +\infty$ as $m\to\infty$.
  \item[(G4)] $J$ is PS-attracting.
\end{itemize}
  Then $J$ has infinitely many geometrically  distinct critical points.
\end{thm}

Here, $(G_4)$ is some sort of compactness property that we shall make precise in
Proposition~\ref{prop:PScontracting}. It is responsible for our restriction to periodic or vanishing
nonlinearities in Theorem~\ref{thm:main}.

\begin{prop}[Local compactness]\label{prop:LocalCompactness}
   Assume $(A')_{per}$ or $(A')_{van}$.  Then for any PS-sequence $(P_n)$ of $J$ there is a subsequence
    $(P_{n_j})$ and a critical point $P$ of $J$ such that 
    $$
      P_{n_j}\wto P\quad\text{in }X,\qquad
      P_{n_j}\to P \quad\text{in }L^{p'}_{\loc}(\R^3;\R^3).  
    $$
  \end{prop}
  \begin{proof}
  Let $(P_n)$ be a PS-sequence $(P_n)$. Then $(P_n)$ is bounded because 
  $$
    O(1)+o(1)\|P_n\| 
    = J(P_n)-\frac{1}{2}J'(P_n)[P_n]
    = \int_{\R^3} \Psi(x,P_n) - \frac{1}{2}\psi(x,P_n)\cdot P_n \,dx
    \gtrsim \min\{ \|P_n\|^{p'},\|P_n\|^{q'}\}.
  $$
  Since $X$ is reflexive by Proposition~\ref{lem:ReflexiveBanachSpace}, we can pass to some weakly convergent
  subsequence still denoted by $(P_n)$ with weak limit $P\in X$. From $J'(P_n)\to 0$ and $P_n\wto P$ we infer, for any given bounded ball
  $B\subset\R^3$ and some $c>0$,
   \begin{align*}
      o(1) 
      &= J'(P_n)[(P_n-P)\ind_B]-J'(P)[(P_n-P)\ind_B] \\
      &= \int_B \big(\psi(x,P_n)-\psi(x,P)\big)\cdot (P_n-P)\,dx   
      -  \int_B (P_n-P) \cdot \mathcal L_\mu^{-1}(P_n-P)\,dx \\
      &\geq c\|\ind_B(P_n-P)\|^2 \min\Big\{(\|P_n\|+\|P\|)^{p'-2},(\|P_n\|+\|P\|)^{q'-2}\Big\}  
       - \int_{\R^3} (P_n-P) \cdot \ind_B \mathcal L_\mu^{-1}(P_n-P)\,dx.   
   \end{align*}  
   Here the last estimate is taken from Proposition~\ref{prop:psi_estimate}.
   The last integral converges to zero because the compact embedding from Proposition~\ref{prop:embeddingH}
   implies  $\ind_B \mathcal L_\mu^{-1}(P_n-P) \to 0$ in $L^p(\R^3;\R^3)$ thanks to $p<6$. 
   Choose $M>0$ such that $\|P_n\|+\|P\|\leq M$ for all $n\in\N$. Then we get for
   $d:=c\min\{M^{p'-2},M^{q'-2}\}$   
   \begin{align*}
      o(1)    
      &\geq d\|\ind_B(P_n-P)\|^2   
      \stackrel{\eqref{eq:def_norm}}\geq d \|\Gamma_1\|_{L^\infty(B)}^{-\frac{2}{p}}   \|P_n-P\|_{L^{p'}(B)}^2
      \qquad\text{as }n\to\infty.
   \end{align*}  
   Since  $B$ was arbitrary, this implies $P_n\to P$ in
   $L^{p'}_{\loc}(\R^3;\R^3)$. From this and $P_n\wto P, J'(P_n)\to 0$ it is
   standard to deduce $J'(P)=0$, so the claim is proved.
  \end{proof}
  
  To prove the PS-contracting property we need the following. 

\begin{prop} \label{prop:convolutionintegral}
  Assume $(A')_{per}$ or $(A')_{van}$.  Then for all $\eps>0$  there are $C_\eps,R_\eps>0$ such that for all
  $g\in X$ the following holds
  $$
    \int_{\R^3} g \cdot \mathcal L_\mu^{-1}  g \,dx
    \leq \eps \|g\|^2 + C_\eps \sup_{y\in \Z^3} \|g(y+\cdot)\|_{L^{q'}(B_{R_\eps}(0))}^2.
  $$
\end{prop}
\begin{proof}  
  We shall use 
  $$
    \int_{\R^3} g \cdot \mathcal L_\mu^{-1}  g \,dx 
    \stackrel{\eqref{eq:Lmu}}\sim \int_{\R^3} g \cdot (-\Delta)^{-1}  g \,dx
    =  \int_{\R^3} g \cdot (K\ast g) \,dx.
  $$
  Now, to estimate the integral on the right, set $K_N(x):= K(x)\ind_{1/N\leq |x|\leq N}$ for $N>1$. 
  We will prove below
  \begin{align} \label{eq:approxconvolutionintegral}
     \int_{\R^3} g \cdot \big((K-K_N)\ast g\big) \,dx
     \les o(1)\|g\|^2 \qquad\text{as }N\to\infty.    
  \end{align}
  So, for any given $\eps>0$ we can find some large enough $N_\eps$ such that
   \begin{equation*}
     \int_{\R^3} g \cdot \big((K-K_{N_\eps})\ast g\big) \,dx \leq \frac{\eps}{2} \|g\|^2
     \qquad\text{for all }g\in X.    
  \end{equation*}
  On the other hand, as
  in the proof of \cite[Lemma 3.2]{EveqWeth_Dual} we find that there are $R_\eps,C_\eps>0$ satisfying
  \begin{align*}
    \int_{\R^3} g \cdot (K_{N_\eps}\ast g) \,dx  
    &\leq   \|K_{N_\eps}\|_\infty  \int_{\R^3} \int_{\R^3} |g(x)||g(y)| \ind_{1/N_\eps\leq|x-y|\leq N_\eps}
    \,dx \,dy
    \\
    &\les_\eps   \|g\|_{q'}^{q'} \sup_{y\in\Z^3} \|g(y+\cdot)\|_{L^{q'}(B_{R_\eps}(0))}^{2-{q'}} \\
    &\les_\eps  \|g\|^{q'} \sup_{y\in\Z^3} \|g(y+\cdot)\|_{L^{q'}(B_{R_\eps}(0))}^{2-{q'}} \\
    &\leq  \frac{\eps}{2} \|g\|^2 + C_\eps  \sup_{y\in\Z^3}
    \|g(y+\cdot)\|_{L^{q'}(B_{R_\eps}(0))}^2\qquad\text{for all }g\in X.
  \end{align*}
  Combining both estimates gives the claim.
  
  \medskip
  
  It remains to prove~\eqref{eq:approxconvolutionintegral}. We shall derive this from Young's convolution
  inequality in weak Lebesgue spaces. We start with the case $q\geq 6$. Then we have,
  as $N\to\infty$,
  \begin{align*}
     \int_{\R^3} g \cdot \big((K-K_N)\ast g\big) \,dx
     &\les \int_{\R^3} |g| \big( (|\cdot|^{-1}\ind_{|\cdot|\leq 1/N})\ast  |g|\big)\,dx 
      + \int_{\R^3} |g| \big( (|\cdot|^{-1}\ind_{|\cdot|\geq N})\ast  |g|\big)\,dx \\
     &\les
     \sup_{|z-y|\leq 1/N} |\Gamma_1(z)\Gamma_1(y)|^{\frac{1}{p}}
      \int_{\R^3} |\Gamma_1^{-\frac{1}{p}} g| \Big( (|\cdot|^{-1}\ind_{|\cdot|\leq 1/N})\ast
      |\Gamma_1^{-\frac{1}{p}}  g|\Big)\,dx \\
     &+ \sup_{|z-y|\geq N} |\Gamma_2(z)\Gamma_2(y)|^{\frac{1}{q}}
      \int_{\R^3} |\Gamma_2^{-\frac{1}{q}} g| \Big( (|\cdot|^{-1}\ind_{|\cdot|\geq N})\ast
      |\Gamma_2^{-\frac{1}{q}} g|\Big)\,dx \\
     &\les \|\Gamma_1\|_\infty^{\frac{2}{p}} 
     \big\| |\cdot|^{-1}\ind_{|\cdot|\leq 1/N}\big\|_{L^{\frac{p}{2},\infty}(\R^3)} \|\Gamma_1^{-\frac{1}{p}}
     g\|_{p'}^2    \\
     &+ \|\Gamma_2\|_\infty^{\frac{1}{q}} \sup_{|z|\geq N/2}
     |\Gamma_2(z)|^{\frac{1}{q}} \cdot \big\||\cdot|^{-1}\ind_{|\cdot|\geq
     N}\big\|_{L^{\frac{q}{2},\infty}(\R^3)}\|\Gamma_2^{-\frac{1}{q}} g\|_{q'}^2 \\
     &\les o(1)  \|\Gamma_1^{-\frac{1}{p}} g\|_{p'}^2
      + o(1) \|\Gamma_2^{-\frac{1}{q}} g\|_{q'}^2 \\
      &\les o(1)  \|g\|^2.  
  \end{align*}
  In the second last estimate we used  $p<6$ as well as $q>6$ or $q=6, \Gamma_2(x)\to 0$ as $|x|\to\infty$,
  cf.~\eqref{eq:van}. To prove the counterpart for $q<6$ we choose $r>6$ such that $\Gamma_2\in
  L^{\frac{r}{r-q}}(\R^3)$. Then, as $N\to\infty$,
  \begin{align*}
     \int_{\R^3} g \cdot \big((K-K_N)\ast g\big) \,dx
     &\les \int_{\R^3} |g| \big( (|\cdot|^{-1}\ind_{|\cdot|\leq 1/N})\ast  |g|\big)\,dx 
      + \int_{\R^3} |g| \big( (|\cdot|^{-1}\ind_{|\cdot|\geq N})\ast  |g|\big)\,dx \\
     &\les o(1) \|g\|^2  + \big\||\cdot|^{-1}\ind_{|\cdot|\geq
     N}\big\|_{L^{\frac{r}{2},\infty}(\R^3)}\|g\|_{r'}^2\\
     &\les o(1) \|g\|^2  + \big\||\cdot|^{-1}\ind_{|\cdot|\geq N}\big\|_{L^{\frac{r}{2},\infty}(\R^3)}
     \|\Gamma_2^{-\frac{1}{q}}g\|_{q'}^2\|\Gamma_2^{\frac{1}{q}}\|_{\frac{rq}{r-q}}^2\\
     &\les o(1) \|g\|^2  + \big\||\cdot|^{-1}\ind_{|\cdot|\geq N}\big\|_{L^{\frac{r}{2},\infty}(\R^3)}
     \|\Gamma_2\|_{\frac{r}{r-q}}^{\frac{2}{q}}   \|g\|^2 \\
     &\sim o(1) \|g\|^2.   
  \end{align*}
  So \eqref{eq:approxconvolutionintegral} is proved.  
\end{proof} 

\begin{prop}\label{prop:PScontracting}
   Assume $(A')_{per}$ or $(A')_{van}$.  Then  $J$ is PS-contracting, i.e., for any given Palais-Smale sequences
   $(P_n)_n,(Q_n)_n$ of $J$ we have $\|P_n -Q_n \| \to 0$ as $n \to \infty$ or 
  $$
    \limsup_{n \to \infty}\|P_n-Q_n \| \geq \kappa
    \quad\text{where }\kappa:= \inf\big\{\|P-Q\|: J'(P)=J'(Q)=0, P\neq Q\big\}.  
  $$
\end{prop}
\begin{proof}
  Let $(P_n),(Q_n)$ be PS-sequences for $J$ such that $\|P_n-Q_n\|\geq \sigma>0$. We first show that
  there are $y_n\in\Z^3$ and $\delta,R>0$ such that  
  \begin{equation}\label{eq:nonvanishing}
     \liminf_{n\to\infty} \|(P_n-Q_n)(y_n+\cdot)\|_{L^{q'}(B_R(0))}\geq  \delta.
  \end{equation} 
  Recalling from Proposition~\ref{prop:LocalCompactness} that PS-sequences are bounded we set  $M:=
  \sup_{n\in\N} (\|P_n\|+\|Q_n\|) <\infty$.  
  On the one hand we have by Proposition~\ref{prop:psi_estimate}
  \begin{align*}
    &\;\int_{\R^3} (P_n-Q_n)\cdot \mathcal L_\mu^{-1}(P_n-Q_n) \,dx \\ 
    &= \int_{\R^3} \big(\psi(x,P_n)-\psi(x,Q_n)\big)\cdot (P_n-Q_n) \,dx  + o(1) \\
    &\gtrsim \|P_n-Q_n\|^2 \min\Big\{(\|P_n\|+\|Q_n\|)^{p'-2},(\|P_n\|+\|Q_n\|)^{q'-2}\Big\} + o(1) \\
    &\geq \sigma^2 \min\{M^{p'-2},M^{q'-2}\} +o(1). 
  \end{align*}
  On the other hand,  Proposition~\ref{prop:convolutionintegral} yields, for
  any given $\eps>0$,
  \begin{align*} 
    \int_{\R^3}(P_n-Q_n)\cdot \mathcal L_\mu^{-1}(P_n-Q_n) \,dx      
    &\leq \eps \|P_n-Q_n\|^2 + C_\eps   \sup_{y\in\Z^3} \|(P_n-Q_n)(y+\cdot)\|_{L^{q'}(B_{R_\eps}(0))}^2 \\ 
    &\leq \eps M^2 + C_\eps   \sup_{y\in\Z^3} \|(P_n-Q_n)(y+\cdot)\|_{L^{q'}(B_{R_\eps}(0))}^2
  \end{align*}
  for some large enough $C_\eps,R_\eps>0$ that are independent of $n$.
  Choose $\eps=\frac{1}{4} M^{-2}\sigma^2 \min\{M^{p'-2},M^{q'-2}\}$. Then  $C:=C_\eps,R:=R_\eps$ leads to
  $$
    \frac{1}{2} \sigma^2 \min\{ M^{p'-2},  M^{q'-2}\} 
    \leq C \sup_{y\in\Z^3}  \|(P_n-Q_n)(y+\cdot)\|_{L^{q'}(B_{R}(0))}^2 \qquad\text{for almost all
    }n\in\N.
  $$
  So, a reasonable choice for   $y=y_n\in\Z^3$  gives \eqref{eq:nonvanishing} where 
  $\delta:= \frac{1}{4C} \sigma^2 \min\{ M^{p'-2},M^{q'-2}\}$. 
  
  \medskip
  
  Having just verified  \eqref{eq:nonvanishing} we first finish the argument assuming $(A')_{van}$. We show
  that $(y_n)$ must be bounded in this case.  Indeed, one of the coefficient functions $\Gamma_1,\Gamma_2$
  decays to zero at infinity, cf.~\eqref{eq:van}.  By definition of the norm and $p'\geq q'$ we get
  \begin{align*}
    M
    &\geq \|P_n-Q_n\| \\ 
    &\geq  \|\Gamma_1^{-\frac{1}{p}} (P_n-Q_n)(y_n+\cdot)\|_{L^{p'}(B_R(0))} + \|\Gamma_2^{-\frac{1}{q}}
    (P_n-Q_n)(y_n+\cdot)\|_{L^{q'}(B_R(0))}\\
    &\gtrsim \|\Gamma_1^{-\frac{1}{p}} (P_n-Q_n)(y_n+\cdot)\|_{L^{q'}(B_R(0))} + \|\Gamma_2^{-\frac{1}{q}}
    (P_n-Q_n)(y_n+\cdot)\|_{L^{q'}(B_R(0))}\\
    &\stackrel{\eqref{eq:nonvanishing}}\geq  \delta
    \Big(\|\Gamma_1\|_{L^\infty(B_R(y_n))}^{-\frac{1}{p}}+\|\Gamma_2\|_{L^\infty(B_R(y_n))}^{-\frac{1}{q}}\Big),
  \end{align*}
  so $(y_n)$ must be bounded in view of~\eqref{eq:van}. Hence, readjusting $R$ allows us to  assume  w.l.o.g.
  that \eqref{eq:nonvanishing} holds for $y_n=0$. Denote by $P,Q$ the weak limits of $(P_n),(Q_n)$ such that
  $P_n\to P$ and $Q_n\to Q$ in $L^{p'}_{\loc}(\R^3;\R^3)$ and $J'(P)=J'(Q)=0$, cf. 
  Proposition~\ref{prop:LocalCompactness}. Then 
  $$ 
    \|P-Q\|_{L^{p'}(B_R(0))} 
    = \lim_{n\to\infty} \|P_n-Q_n\|_{L^{p'}(B_R(0))} 
    \stackrel{\eqref{eq:nonvanishing}}\geq \delta>0
  $$
  and in particular $P\neq Q$. This finally implies, by weak lower semicontinuity of the norm in $X$, 
  $$
    \lim_{n\to\infty} \|P_n-Q_n\|  
    \geq \|P-Q\| 
    \geq \inf\Big\{ \|\tilde P-\tilde Q\| : J'(\tilde P)=J'(\tilde
    Q)=0,\tilde P\neq \tilde Q\Big\}, 
  $$
  which is all we had to show. 
  
  \medskip
  
  In the case $(A')_{per}$ the translational $\Z^3$-invariance of the functional implies that
  $(P_n(y_n+\cdot))$ and $(Q_n(y_n+\cdot))$ are Palais-Smale sequences of $J$ as well. So
  Proposition~\ref{prop:LocalCompactness} gives, after passing to a suitable subsequence,   weak  convergence
  to some critical points $P,Q\in X$, respectively, and $$
    P_n(y_n+\cdot) \to P, \qquad Q_n(y_n+\cdot)\to Q \quad\text{in }L^{p'}(B_R(0)).
  $$ 
  Then $P\neq Q$ follows from  
  $$
    \|P-Q\|_{L^{p'}(B_R(0))} 
    = \lim_{n\to\infty} \|(P_n-Q_n)(y_n+\cdot)\|_{L^{p'}(B_R(0))}\geq \delta>0.
  $$
  From this we may conclude as above.  
\end{proof}

  \begin{thm}\label{thm:dualR3}
    Assume $(A')_{per}$ or $(A')_{van}$.  Then $J$ has   infinitely many geometrically
    distinct critical points. 
  \end{thm}
  \begin{proof}
    We verify the hypotheses of Theorem~\ref{thm:CPTheorem}. $(G1)$ is immediate and $(G2)$ follows from
    Proposition~\ref{prop:J1}.
    As to $(G_3)$, for any given $m\in\N$ one may for instance choose $X_m:=\spa\{\nabla\times\phi_1,\ldots,\nabla\times\phi_m\}\subset X$
    where $\supp(\phi_j)\subset [j,j+1]$ is such that $\phi_j$ is not a gradient. The latter ensures
    $\nabla\times\phi_j\neq 0$ and hence the linear independence of
    $\{\nabla\times\phi_1,\ldots,\nabla\times\phi_m\}$ follows from the support property. Hence,  
    $X_m\subset X$ with $\dim(X_m)=m$. Clearly, Proposition~\ref{prop:estimates} gives $J(P)\to -\infty$
    uniformly as $P\in X_m,\|P\|\to\infty$, so $(G_3)$ holds as well. Finally, $(G_4)$ holds by
    Proposition~\ref{prop:PScontracting}.  So Theorem~\ref{thm:CPTheorem} proves the claim.
  \end{proof}
   
 \subsection{Existence of a ground state} 
      
  We construct a ground state as a limit of nontrivial critical points of $J$. We start with a simple
  observation.
  
  \begin{prop} \label{prop:GSProp}
    Assume $(A')_{per}$ or $(A')_{van}$. There are $\kappa_1,\kappa_2,\kappa_3>0$ such that for all $P\in X\sm\{0\}$ satisfying $J'(P)=0$ we
    have $$
      \|P\|\geq \kappa_1,\qquad \int_{\R^3} \mathcal L_\mu^{-1}P\cdot P\,dx \geq \kappa_2,\qquad J(P)\geq \kappa_3.
    $$
  \end{prop}
  \begin{proof}
    Any nontrivial critical point satisfies  $J'(P)[P]=0$, i.e.,
    $$
      \int_{\R^3} \psi(x,P)\cdot P \,dx =  \int_{\R^3} \mathcal L_\mu^{-1}P\cdot P\,dx.
    $$
    Hence,  
   $$
      \min\{\|P\|^{p'},\|P\|^{q'}\} 
     \les   \int_{\R^3} \psi_0(x,|P|)|P|\,dx
     = \int_{\R^3} \psi(x,P)\cdot P  \,dx
     =  \int_{\R^3} \mathcal L_\mu^{-1}P\cdot P\,dx 
     \les \|P\|^2.
   $$
   From this and $p',q'<2$ we deduce   $\|P\|\geq \kappa_1$ for some $\kappa_1>0$. The lower bound
   for the integral involving $\mathcal L_\mu^{-1}$ is then a direct consequence of the chain of inequalities
   above. Finally, 
   $$
     J(P)
     = J(P)-\frac{1}{2}J'(P)[P]
     = \int_{\R^3} \Psi(x,P)-\frac{1}{2}\psi(x, P)\cdot P\,dx 
     \gtrsim   \min\{\|P\|^{p'},\|P\|^{q'}\}
     \gtrsim  \min\{\kappa_1^{p'},\kappa_1^{q'}\},
   $$
   which finishes the proof.
  \end{proof}

 \begin{thm} \label{thm:dual}
   Assume $(A')_{per}$ or $(A')_{van}$. Then $J$ admits a ground state.
 \end{thm}
  \begin{proof}
   Let $(P_n)$ be a sequence of nontrivial critical points of $J$ over $X$ such that 
   $$
     J(P_n)\to \inf \big\{J(\tilde P): J'(\tilde P)=0, \tilde P\in X\sm\{0\}\big\}.
   $$ 
   Such a sequence exists in view of Theorem~\ref{thm:dualR3}.
   Then $(P_n)$ is a PS-sequence and Proposition~\ref{prop:LocalCompactness} provides a subsequence, still
   denoted by $(P_n)$, that  converges weakly to a  $P\in X$ with $J'(P)=0$ and $P_n\to P$ in
   $L^{p'}_{\loc}(\R^3;\R^3)$. Choose $M>0$ such that $\|P\|+\|P_n\|\leq M$.
   Then Proposition~\ref{prop:convolutionintegral} implies, for any given $\eps>0$
   \begin{align*}
     \left|\int_{\R^3} \mathcal L_\mu^{-1} P_n\cdot P_n\,dx - \int_{\R^3}\mathcal L_\mu^{-1}P\cdot
     P\,dx\right| &= \left|\int_{\R^3} \mathcal L_\mu^{-1}(P_n-P)\cdot (P_n-P)\,dx + o(1) \right| \\ 
    &\leq \eps \|P_n-P\|^2 +  o(1)
    \leq \eps M^2 +  o(1)
   \end{align*}
   and thus
   \begin{equation}\label{eq:convergencevolutionintegral}
     \int_{\R^3} \mathcal L_\mu^{-1}P\cdot P\,dx
     = \lim_{n\to\infty} \int_{\R^3} \mathcal L_\mu^{-1} P_n\cdot P_n\,dx. 
   \end{equation} 
   In particular, the integral on the left is bounded from below by  $\kappa_2>0$ as in
   Proposition~\ref{prop:GSProp}, so   $P\neq 0$. We now use that $J_1$ is weakly lower
   semicontinuous, cf.~Proposition~\ref{prop:psi_estimate}. We thus find
   \begin{align*}
     J(P)
     &= J_1(P) -  \frac{1}{2}  \int_{\R^3} \mathcal L_\mu^{-1}P\cdot P\,dx  \\
     &\leq \liminf_{n\to\infty} J_1(P_n) - \frac{1}{2} \int_{\R^3} \mathcal L_\mu^{-1}P\cdot P\,dx  \\
     &\stackrel{\eqref{eq:convergencevolutionintegral}}= \liminf_{n\to\infty} \Big[\;J_1(P_n) - 
     \frac{1}{2}\int_{\R^3} \mathcal L_\mu^{-1} P_n\cdot P_n\,dx\;\Big]
     \\
     &= \liminf_{n\to\infty} J(P_n) \\
     &= \inf \big\{J(\tilde P): J'(\tilde P)=0, \tilde P\in X\sm\{0\}\big\}.
   \end{align*}
   Hence, $P$ is a ground state, which is all we had to show.
 \end{proof}

\section{Conclusion - proof of Theorem~\ref{thm:main}} \label{sec:proof}

We have to prove that the functional $I$ has infinitely many geometrically distinct points as well as a ground
states. By Lemma~\ref{lem:equivalence}, the first claim is proved once we have shown that the functional $J$
from \eqref{eq:functionalJdual} has infinitely many geometrically distinct critical points. Here, in view of 
Proposition~\ref{prop:Apsi} and assumption
$(A)_{per}$ or $(A)_{van}$, we know that $\psi(x,\cdot)=f(x,\cdot)^{-1}$ exists and satisfies
 $(A')_{per}$ or $(A')_{van}$, respectively. So infinitely many geometrically
distinct points of $J$ exist by Theorem~\ref{thm:dualR3}.
These provide infinitely many geometrically distinct critical points of $I$ by Lemma~\ref{lem:equivalence}.

\medskip

The existence of a ground state $P^\star\in X\sm \{0\}$ of $J$ was proved in Theorem~\ref{thm:dual}, so
Lemma~\ref{lem:equivalence} and the same argument as in \cite[Theorem~15]{Man_Nonlocal} prove that this
ground state  produces a ground state $E^\star$ of $I$ via $E^\star:= \psi(x,P^\star)$.
So we conclude that $I$ admits a ground state as well. This finishes the proof. \qed

\bibliographystyle{abbrv}
\bibliography{biblio}

\end{document}